\documentclass{amsart}
\usepackage{amsfonts}
\usepackage{amsmath}
\usepackage{amssymb}
\usepackage{amsthm}

\newtheorem{theorem}{Theorem}
\newtheorem{prop}{Proposition}
\newtheorem{lemma}{Lemma}
\newtheorem{rem}{Remark}

\newtheorem{exmp}{Example}

\begin{document}
\author{Mark Pankov}
\title[Characterizations of strong semilinear embeddings]
{Characterizations of strong semilinear embeddings in terms of general linear and projective linear groups}
\address{Department of Mathematics and Computer Science, University of Warmia and Mazury,
S{\l}oneczna 54, 10-710 Olsztyn, Poland}
\email{pankov@matman.uwm.edu.pl}

\maketitle

\begin{abstract}
Let $V$ and $V'$ be vector spaces over division rings.
Suppose $\dim V$ is finite and not less than $3$. 
Consider a mapping $l:V\to V$ with the following property: 
for every $u\in {\rm GL}(V)$ there is $u'\in {\rm GL}(V')$ such that $lu=u'l$.
Our first result states that $l$ is a strong semilinear embedding if 
$l|_{V\setminus\{0\}}$ is non-constant and 
the dimension of the subspace of $V'$ spanned by $l(V)$ is not greater than $n$.
We present examples showing that these conditions can not be omitted.
In some special cases, this statement can be obtained from \cite{DH,Zha}.
Denote by ${\mathcal P}(V)$ the projective space associated with $V$
and consider the mapping $f:{\mathcal P}(V)\to {\mathcal P}(V')$ with the following property: 
for every $h\in {\rm PGL}(V)$ there is $h'\in {\rm PGL}(V')$ such that $fh=h'f$.
By the second result,
$f$ is induced by a strong semilinear embedding of $V$ in $V'$
if $f$ is non-constant and its image is contained in a subspace of $V'$
whose dimension is not greater than $n$, we also require that $R'$ is a field.
\end{abstract}

\section{Introduction and statement of results}
Let $V$ and $V'$ be left vector spaces over division rings $R$ and $R'$, respectively.
Suppose that $\dim V=n$  is finite.

A mapping $l:V\to V'$ is called {\it semilinear} if
$$l(x+y)=l(x)+l(y)\;\;\;\;\;\forall\;x,y\in V$$
and there is a homomorphism $\sigma:R\to R'$ such that
$$l(ax)=\sigma(a)l(x)\;\;\;\;\;\forall\;a\in R,x\in V.$$
If $l$ is non-zero then there exists precisely one non-zero homomorphism $\sigma$ satisfying this condition
and the mapping $l$ is said to be $\sigma$-{\it linear}.
In the case when $R=R'$ and $\sigma$ is identity, the mapping $l$ is linear.
If $l$ is $\sigma$-linear then for every non-zero scalar $a\in R'$
the mapping $al$ is semilinear over the homomorphism $b\to a\sigma(b)a^{-1}$.

All non-zero homomorphisms of division rings are injective.
Each homomorphism of ${\mathbb R}$ to itself is zero or identity and 
every semilinear mapping between real vector spaces is linear.
However, there are infinitely many automorphisms of ${\mathbb C}$
and there exist non-zero non-surjective homomorphisms of ${\mathbb C}$ to itself,
see \cite[p. 114]{Benz} and \cite{PBYale}.

A {\it semilinear isomorphism} is a semilinear bijection such that 
the associated homomorphism of division rings is an isomorphism.
Note that there exist semilinear bijections over non-surjective homomorphisms
(for example, the natural bijection of ${\mathbb R}^{2n}$ to ${\mathbb C}^{n}$).
A semilinear injection $l:V\to V'$ is called a {\it semilinear $k$-embedding} if
it transfers any $k$ linearly independent vectors to linearly independent vectors.
In the case when $k=n$, we say that $l$ is a {\it strong semilinear embedding}.
Examples of non-strong semilinear embeddings can be found in 
\cite{BR,Havlicek-sur,Kreuzer1}.

Every strong semilinear embedding $l:V\to V'$ satisfies the following condition:
\begin{enumerate}
\item[(GL)] for every $u\in {\rm GL}(V)$ there is $u'\in {\rm GL}(V')$ such that $lu=u'l$.
\end{enumerate}
Indeed, if $u\in {\rm GL}(V)$ and $x_{1},\dots,x_{n}$ form a base of $V$
then $l(x_{1}),\dots,l(x_{n})$ are linearly independent and the same holds for $lu(x_{1}),\dots,lu(x_{n})$;
an easy verification shows that the equality $lu=u'l$ holds for  any $u'\in {\rm GL}(V')$ transferring every $l(x_{i})$ to $lu(x_{i})$.

We say that an arbitrary mapping $l:V\to V'$ is a GL-{\it mapping} if it satisfies the condition (GL).

Suppose that $x'\in V'$ and $g:V\to V'$ is a mapping such that $g(x)=x'$ for every $x\in V\setminus\{0\}$
(we do not require that $g(0)=x'$ or $g(0)=0$).
If $u\in {\rm GL}(V)$ then the equality  $gu=u'g$ holds, for example, for $u'={\rm id}_{V'}$.
Thus $g$ is a GL-mapping.

A GL-mapping of $V$ to $V'$ is said to be {\it non-trivial} if its restriction to $V\setminus \{0\}$ is non-constant.

For every mapping $g:V\to V'$ and every subspace $S\subset V$ we denote by $S_{g}$
the subspace of $V'$ spanned by $g(S)$,
in particular, we write $V_{g}$ for the subspace of $V'$ spanned by the image of $g$.

Our first result is the following characterization of strong semilinear embeddings.

\begin{theorem}\label{theorem-main1}
Let  $n\ge 3$ and let $g:V\to V'$ be a non-trivial {\rm GL}-mapping.
If $\dim V_{g}\le n$ then $g$ is a strong semilinear embedding.
\end{theorem}

In Section 6 we present an example showing that 
the condition $\dim V_{g}\le n$ in Theorem \ref{theorem-main1} cannot be omitted.
The proof of Theorem \ref{theorem-main1} (Section 3) is based on the description of all 
finite subset $X\subset V$  such that every permutation on $X$ can be extended to a linear 
automorphism of $V$ (Section 2). 

\begin{rem}\label{rem}{\rm
If $g:V\to V'$ is a non-trivial GL-mapping then 
for every $u\in {\rm GL}(V)$ there is unique $\overline{u}\in {\rm GL}(V_{g})$ satisfying $gu={\overline u}g$  
and the mapping $u\to {\overline u}$ is a homomorphism of ${\rm GL}(V)$ to ${\rm GL}(V_{g})$
(Section 3).
If $\dim V=\dim V'\ge 2$, $R$ and $R'$ are of the same characteristic and
$R'$ is a finite-dimensional vector space over its center then
all homomorphisms of ${\rm SL}(V)$ to ${\rm GL}(V')$ are determined \cite{DH,Zha}.
Theorem \ref{theorem-main1}  follows from the results of \cite{DH,Zha} 
in the case when $\dim V_{g}=n$, $R$ and $R'$ are of the same characteristic and
$R'$ is a finite-dimensional vector space over its center.
}\end{rem}

Let ${\mathcal P}(V)$ be the projective space associated with $V$
--- the points are $1$-dimen\-sional subspaces of $V$ and the lines are defined by $2$-dimensional subspaces of $V$
(in the case when $n=2$, this is a projective line).
Every semilineaer injection $l:V\to V'$ induces the mapping
$$\pi(l):{\mathcal P}(V)\to {\mathcal P}(V')$$
which transfers every $1$-dimensional subspace $P\subset V$
to $P_{l}$. For every non-zero scalar $a\in R'$ we have $\pi(al)=\pi(l)$.
Conversely, if $s:V\to V'$ is a semilinear injection such that $\pi(l)=\pi(s)$
and the image of $l$ contains two linearly independent vectors then $s$ is a scalar multiple of $l$
\cite[Subsection 1.3.2]{Pankov1}.

Recall that ${\rm PGL}(V)$ is the group formed by all transformations of ${\mathcal P}(V)$ 
induced by linear automorphisms of $V$.
There is the natural group homomorphism 
$$\pi:{\rm GL}(V)\to{\rm PGL}(V)$$
whose kernel consists of all homothecies $x\to ax$ with $a$ belonging to the center of $R$.

For every strong semilinear embedding $l:V\to V'$ the mapping $f=\pi(l)$ satisfies the following condition:
\begin{enumerate}
\item[(PGL)] for every $h\in {\rm PGL}(V)$ there is $h'\in {\rm PGL}(V')$ such that $fh=h'f$.
\end{enumerate}
This easy follows from the fact that $l$ is a GL-mapping.
Also, this condition holds for  any constant mapping $f:{\mathcal P}(V)\to {\mathcal P}(V')$
(if $f(P)=P'$ for every $P\in {\mathcal P}(V)$ and $h\in {\rm PGL}(V)$ 
then we have $fh=h'f$ for any $h'\in {\rm PGL}(V')$ leaving fixed $P'$).

Every mapping $f:{\mathcal P}(V)\to {\mathcal P}(V')$ satisfying the condition (PGL)
will be called a PGL-{\it mapping}.

For every subset ${\mathcal X}\subset {\mathcal P}(V)$ the minimal subspace of $V$
containing every element of ${\mathcal X}$ is said to be {\it spanned} by ${\mathcal X}$.
For every mapping $f:{\mathcal P}(V)\to {\mathcal P}(V')$ and every subspace $S\subset V$
we denote by $S_{f}$ the subspace of $V'$ spanned by $f({\mathcal P}(S))$
and write $V_{f}$ for the subspace of $V'$ spanned by the image of $f$.

\begin{theorem}\label{theorem-main2}
Let  $n\ge 3$, $R'$ be a field and  $f:{\mathcal P}(V)\to {\mathcal P}(V')$ be a non-constant {\rm PGL}-mapping.
If $\dim V_{f}\le n$ then  $f$ is induced by a strong semilinear embedding.
\end{theorem}

The proof (Section 5) is similar to the proof of Theorem \ref{theorem-main1}.
We can determine all finite subsets ${\mathcal X}\subset {\mathcal P}(V)$
such that every permutation on ${\mathcal X}$ is extendable to an element of ${\rm PGL}(V)$
only in the case when $R$ is field (Section 4).
By this reason, in Theorem \ref{theorem-main2} we require that $R'$ is field.

\section{Extendability of permutations on finite subsets of vector spaces}
Let $X$ be a finite subset of $V$ containing more than one vector.
Denote by $S(X)$ the group of all permutations on $X$.
We want to determine all cases when every element of $S(X)$ can be extended to a linear automorphism of $V$.
This is possible, for example, if $X$ is formed by linearly independent vectors.

\begin{exmp}\label{exmp-p1-1}{\rm
Suppose that $X$ consists of linearly independent vectors $x_{1},\dots,x_{m}$ and
the vector
$$x_{m+1}=-(x_{1}+\dots+x_{m}).$$
For every $i\in \{1,\dots m-1\}$ we take any linear automorphism $u_{i}\in {\rm GL}(V)$ such that
$$u_{i}(x_{i})=x_{i+1},\;u_{i}(x_{i+1})=x_{i}\;\mbox{ and }\;u_{i}(x_{j})=x_{j}\;\mbox{ if }\;j\ne i,i+1,m+1.$$
Every $u_i$ sends $x_{m+1}$ to itself.
Consider a linear automorphism $v\in {\rm GL}(V)$ leaving fixed every $x_{i}$ for $i\le m-1$ and
transferring $x_{m}$ to $x_{m+1}$.
Then
$$v(x_{m+1})=-(v(x_{1})+\dots+v(x_{m}))=-(x_{1}+\dots+x_{m-1}-x_{1}-\dots-x_{m-1}-x_{m})=x_{m}.$$
So, all transpositions of type $(x_{i},x_{i+1})$ can be extended to linear automorphisms of $V$.
Since $S(X)$ is spanned by these transpositions,
every permutation on $X$ is extendable to a linear automorphism of $V$.
}\end{exmp}

\begin{prop}\label{prop1}
If every permutation on $X$ can be extended to a linear automorphism of $V$
then $X$ is formed by linearly independent vectors or it consists of
$$x_{1},\dots,x_{m},-(x_{1}+\dots+x_{m}),$$
where $x_{1},\dots,x_{m}$ are linearly independent.
\end{prop}

\begin{proof}
Let $x_{1},\dots,x_{k}$ be the elements of $X$.
Suppose that these vectors are not linearly independent
and consider any maximal collection of linearly independent vectors from $X$.
We can assume that this collection is formed by $x_{1},\dots,x_{m}$, $m<k$.
Then every $x_{p}$ with $p>m$ is a linear combination of $x_{1},\dots,x_{m}$, i.e.
$x_{p}=\sum^{m}_{l=1}a_{l}x_{l}$.
Let $u\in {\rm GL}(V)$ be an extension of the transposition $(x_{i},x_{j})$, $i,j\le m$.
Then
$$u(x_{i})=x_{j},\;u(x_{j})=x_{i}\;\mbox{ and }\;u(x_{l})=x_{l}\;\mbox{ if }\;l\ne i,j.$$
We have
$$\sum^{m}_{l=1}a_{l}x_{l}=x_{p}=u(x_{p})=\sum^{m}_{l=1}b_{l}x_{l},\;\mbox{ where }\;
b_{i}=a_{j},\;b_{j}=a_{i}\;\mbox{ and }\;b_{l}=a_{l}\;\mbox{ if }\;l\ne i,j.$$
Since $x_{1},\dots,x_{m}$ are linearly independent, the latter means that $a_{i}=a_{j}$.
This equality holds for any $i,j\le m$ and we have
$$x_{p}=a(x_{1}+\dots+x_{m})$$
for some non-zero scalar $a\in R$.
Let $v\in {\rm GL}(V)$ be an extension of the transposition $(x_{1},x_{p})$.
Then
$$v(x_{1})=x_{p},\;v(x_{p})=x_{1}\;\mbox{ and }\;v(x_{i})=x_{i}\;\mbox{ if }\;i\ne 1,p$$
We have
$$x_{1}=v(x_{p})=a(v(x_{1})+\dots+v(x_{m}))=a(x_{p}+x_{2}+\dots+x_{m})=$$
$$=a^{2}(x_{1}+\dots+ x_{m})+a(x_{2}+\dots+x_{m})=a^{2}x_{1}+(a^{2}+a)(x_{2}+\dots+x_{m}).$$
Hence $a^{2}=1$ and $a^{2}+a=0$ which implies that $a=-1$ and
$$x_{p}=-(x_{1}+\dots+x_{m}).$$
This equality holds for every $p>m$.
Therefore, $k=m+1$ and the second possibility is realized.
\end{proof}

\section{Proof of Theorem \ref{theorem-main1}}
Let $n\ge 2$ and let $g:V\to V'$ be a non-trivial GL-mapping.

\begin{lemma}\label{lemma-t1-1}
The following assertions are fulfilled:
\begin{enumerate}
\item[{\rm (1)}] for every $u\in {\rm GL}(V)$ there exists unique $\overline{u}\in {\rm GL}(V_{g})$ such that
$gu={\overline u}g$,
\item[{\rm (2)}] the mapping $u\to \overline{u}$ is a homomorphism of ${\rm GL}(V)$ to ${\rm GL}(V_{g})$.
\end{enumerate}
\end{lemma}

\begin{proof}
(1). Let $u\in {\rm GL}(V)$.
We take any  $u'\in {\rm GL}(V')$ satisfying $gu=u'g$.
For every $y\in g(V)$ there exists $x\in V$ such that $y=g(x)$ and we have
$$u'(y)=u'g(x)=gu(x)\in g(V).$$
Thus $u'$ transfers $g(V)$ to itself.
This means that $V_{g}$ is invariant for $u'$, since it is spanned by $g(V)$.

Suppose that $u''\in {\rm GL}(V')$ satisfies the condition $gu=u''g$.
If $y=g(x)$, $x\in V$ then
$$u'(y)=u'g(x)=gu(x)=u''g(x)=u''(y).$$
Therefore, $u'|_{g(V)}=u''|_{g(V)}$ which implies that $u'|_{V_{g}}=u''|_{V_{g}}$.

(2). If $u,v\in {\rm GL}(V)$ then
$$\overline{uv}g=guv={\overline u}gv={\overline u}\,{\overline v}g$$
and ${\overline u}\,{\overline v}=\overline{uv}$ by the first statement.
\end{proof}

\begin{lemma}\label{lemma-t1-2}
If $x,y\in V\setminus\{0\}$ are linearly independent then $g(x)\ne g(y)$.
\end{lemma}

\begin{proof}
Suppose that $g(x)=g(y)$ for linearly independent vectors $x,y\in V\setminus\{0\}$.
Then
$$gu(x)={\overline u}g(x)={\overline u}g(y)=gu(y)\;\;\;\;\;\forall\;u\in {\rm GL}(V).$$
Let $z\in V\setminus\{0\}$. In the case when $x,z$ are linearly independent,
we take any $u\in {\rm GL}(V)$ such that $u(x)=x$ and $u(y)=z$. Then
$$g(x)=gu(x)=gu(y)=g(z).$$
If $z$ is a scalar multiple of $x$ then $y,z$ are linearly independent
and the same arguments show that $g(y)=g(z)$.
Thus $g|_{V\setminus\{0\}}$ is constant which contradicts our assumption.
\end{proof}

\begin{lemma}\label{lemma-t1-3}
For every $u\in {\rm GL}(V)$ the following assertions are fulfilled:
\begin{enumerate}
\item[{\rm (1)}] if a subspace $S\subset V$ is invariant for $u$ then $S_{g}$ is invariant for $\overline{u}$,
\item[{\rm (2)}]  $({\rm Ker}({\rm id}_{V}-u))_{g}\subset {\rm Ker}({\rm id}_{V_{g}}-{\overline u})$.
\end{enumerate}
\end{lemma}

\begin{proof}
(1).
Let $S$ be a subspace of $V$ invariant for $u$.
If $y\in g(S)$ then $y=g(x)$, $x\in S$ and
$$\overline{u}(y)=\overline{u}g(x)=gu(x)\in g(S).$$
Therefore, ${\overline u}$ transfers $g(S)$ to itself.
This means that $S_{g}$ is invariant for $\overline{u}$, since it is spanned by $g(S)$.

(2). If $u(x)=x$ then $\overline{u}g(x)=gu(x)=g(x)$.
Thus
$$g({\rm Ker}({\rm id}_{V}-u))\subset {\rm Ker}({\rm id}_{V_{g}}-{\overline u})$$
which implies the required equality.
\end{proof}

\begin{lemma}\label{lemma-t1-4}
If ${\overline u}$ is identity then $u$ is identity or  a homothety.
\end{lemma}

\begin{proof}
If ${\overline u}$ is identity then
$$gu(x)=\overline{u}g(x)=g(x)\;\;\;\;\;\forall\;x\in V.$$
If $x\in V\setminus\{0\}$ and $u(x)$ are linearly independent then, by Lemma \ref{lemma-t1-2},
$gu(x)\ne g(x)$ which contradicts to the latter equality.
Thus $u(x)$ is a scalar multiple of $x$ for every $x\in V\setminus\{0\}$
which is possible only in the case when $u$ is identity or  a homothety.
\end{proof}

We use Proposition \ref{prop1} to  prove the following.

\begin{lemma}\label{lemma-t1-5}
$g(x)\ne 0$ for every $x\in V\setminus\{0\}$
and $g$ transfers any $n-1$ linearly independent vectors to $n-1$ linearly independent vectors.
\end{lemma}

\begin{proof}
Let $X\subset V$ be a subset consisting of $n-1$ linearly independent vectors
and let $B$ be a base of $V$ containing $X$.
Lemma \ref{lemma-t1-2} guarantees that $g|_{B}$ is a bijection to $B':=g(B)$.

For every permutation $s'$ on $B'$ there is a permutation $s$ on $B$ satisfying $gs=s'g|_{B}$.
Let $u\in {\rm GL}(V)$ be an extension of $s$
(since $B$ is a base of $V$, every permutation on $B$ can be extended to a linear automorphism of $V$).
If $y\in B'$ then $y=g(x)$, $x\in B$ and
$${\overline u}(y)={\overline u}g(x)=gu(x)=gs(x)=s'g(x)=s'(y),$$
i.e. ${\overline u}|_{B'}=s'$.
So, every permutation on $B'$ can be extended to a linear automorphism of $V_{g}$.

By Proposition \ref{prop1},
$B'$ is formed by $n$ linearly independent vectors or
$$B'=\{y_{1},\dots,y_{n-1},-(y_{1}+\dots+y_{n-1})\},$$
where $y_{1},\dots,y_{n-1}$ are linearly independent.
In each of these cases, any $n-1$ vectors of $B'$ are linearly independent.
This implies that $g(X)$ is formed by $n-1$ linearly independent vectors.

Since every non-zero vector of $V$ is contained in a subset formed by $n-1$ linearly independent vectors,
we have $g(x)\ne 0$ for every $x\in V\setminus\{0\}$.
\end{proof}

From this moment we suppose that $\dim V_{g}\le n$.

\begin{lemma}\label{lemma-t1-6}
$\dim V_{g}=n$ and for every $(n-1)$-dimensional subspace $S\subset V$
the subspace $S_{g}$ is $(n-1)$-dimensional.
\end{lemma}

\begin{proof}
Let $S$ be an $(n-1)$-dimensional subspace of $V$.
We take any $u\in {\rm GL}(V)$ such that
\begin{equation}\label{eq-t1-1}
{\rm Ker}({\rm id}_{V}-u)=S.
\end{equation}
By the second part of Lemma \ref{lemma-t1-3}, we have
\begin{equation}\label{eq-t1-2}
S_{g}\subset {\rm Ker}({\rm id}_{V_{g}}-\overline{u}).
\end{equation}
Lemma \ref{lemma-t1-5} implies that $\dim S_{g}\ge n-1$.
If $S_{g}$ is $n$-dimensional then
$$\dim({\rm Ker}({\rm id}_{V_{g}}-\overline{u}))\ge n$$
by \eqref{eq-t1-2}.
Since $\dim V_{g}\le n$,
the kernel of ${\rm id}_{V_{g}}-\overline{u}$ coincides with $V_{g}$ which means that $\overline{u}$ is identity.
Then $u$ is identity or  a homothety (Lemma \ref{lemma-t1-4})
which contradicts to \eqref{eq-t1-1}.
Therefore, $S_{g}$ is $(n-1)$-dimensional.
If $\dim V_{g}<n$ then $V_{g}=S_{g}$ and \eqref{eq-t1-2} implies that $\overline{u}$ is identity
which is impossible.
\end{proof}

\begin{lemma}\label{lemma-t1-7}
$\dim S_{g}=\dim S$ for every non-zero subspace $S\subset V$.
\end{lemma}

\begin{proof}
If $n=2$ then the statement follows from the previous lemma.
Suppose that $n\ge 3$ and consider any $(n-1)$-dimensional subspace $U\subset V$.
If $U$ is invariant for $u\in {\rm GL}(V)$ then $U_{g}$ is invariant for $\overline{u}$ 
(the first part of Lemma \ref{lemma-t1-3}).
This means that $g|_{U}$ is a GL-mapping;
it is non-trivial (by Lemma \ref{lemma-t1-2}) and $\dim U_{g}=n-1$ (by Lemma \ref{lemma-t1-6}).
It follows from the arguments given above that $\dim S_{g}=n-2$ for every $(n-2)$-dimensional subspace $S\subset U$.
Step by step, we show that $\dim S_{g}=\dim S$ for every non-zero subspace $S\subset V$.
\end{proof}

\begin{lemma}\label{lemma-t1-8}
$g(0)=0$.
\end{lemma}

\begin{proof}
The condition $\dim V_{g}=n\ge 2$ implies the existence of $1$-dimensional subspaces $P,Q\subset V$ such that
$P_{g}$ and $Q_{g}$ are distinct $1$-dimensional subspaces of $V_{g}$.
Then $g(0)\in P_{g}\cap Q_{g}=0$.
\end{proof}

From this moment we suppose that $n\ge 3$. 
Then $\dim V'\ge 3$ (since $\dim V_{f}=n$ by Lemma \ref{lemma-t1-6}).
For every semilinear injection $l:V\to V'$  the mapping $\pi(l)$ sends every line of ${\mathcal P}(V)$ 
to a subset contained in a line of ${\mathcal P}(V')$. 
We will use the following generalization of the Fundamental Theorem of Projective Geometry.

\begin{theorem}[C. A. Faure, A. Fr\"{o}licher, H. Havlicek \cite{FaureFrolicher,Havlicek}]\label{theoremFFH}
If $f:{\mathcal P}(V)\to {\mathcal P}(V')$ is a mapping transferring every line 
to a subset in a line and the image of $f$ is not contained in a line then
$f$ is induced by a semilinear injection of $V$ to $V'$.
\end{theorem}

Consider the mapping $f$ which transfers every $P\in {\mathcal P}(V)$ to $P_{g}\in {\mathcal P}(V')$.
It sends the line of ${\mathcal P}(V)$ defined  by a $2$-dimensional subspace $S\subset V$
to a subset in the line of ${\mathcal P}(V')$ defined by $S_{g}$.
The condition $\dim V_{g}=n\ge 3$ guarantees that the image of $f$ is not contained in a line.
Therefore, $f$ satisfies the conditions of Theorem \ref{theoremFFH}, i.e. $f$ is induced by a semilinear injection $l:V\to V'$.
Then for every vector $x\in V\setminus\{0\}$ there exists a non-zero scalar $a_{x}\in R'$ such that
$$g(x)=a_{x}l(x).$$
Since $V_{g}=V_{l}$ and $V_{g}$ is $n$-dimensional,
$l$ transfers every base of $V$ to a base of $V_{g}$. 
The latter means that $l$ is a strong semilinear embedding of $V$ in $V_{g}$.

\begin{lemma}\label{lemma-t1-9}
$a_{x}=a_{y}$ for all $x,y\in V\setminus\{0\}$.
\end{lemma}

\begin{proof}
Since $l$  is a GL-mapping, for every $u\in {\rm GL}(V)$ there exists $\tilde{u}\in {\rm GL}(V_{g})$
such that $lu=\tilde{u}l$.
Then
$$a_{x}\overline{u}l(x)=\overline{u}g(x)=gu(x)=a_{u(x)}lu(x)=a_{u(x)}\tilde{u}l(x)$$
and
\begin{equation}\label{eq-t1-3}
\overline{u}l(x)=a^{-1}_{x}a_{u(x)}\tilde{u}l(x)
\end{equation}
for every $x\in V\setminus\{0\}$.
The mappings  $\overline{u}l$ and $\tilde{u}l$ both are semilinear injections
and \eqref{eq-t1-3} guarantees that $\pi(\overline{u}l)=\pi(\tilde{u}l)$.
By \cite[Proposition 1.7]{Pankov1}, there exists a non-zero scalar $b_{u}\in R'$
such that $\overline{u}l=b_{u}\tilde{u}l$. 
Then \eqref{eq-t1-3} implies that 
$$a_{u(x)}=a_{x}b_{u}.$$
Let $x,y\in V\setminus\{0\}$. We take any vector $z\in V$ which does not belong to the subspace
spanned by $x,y$ (since $n\ge 3$, this is possible). There exists $u\in {\rm GL}(V)$ satisfying
$u(x)=y$ and $u(z)=z$.
Then $a_{z}=a_{u(z)}=a_{z}b_{u}$ which means that $b_{u}=1$ and $a_{y}=a_{u(x)}=a_{x}b_{u}=a_{x}$.
\end{proof}

Lemmas \ref{lemma-t1-8} and \ref{lemma-t1-9} show that $g$ is a scalar multiple of 
the strong semilinear embedding $l$, i.e.
$g$ is a strong semilinear embedding.

\section{Extendability of permutations on finite subsets of projective spaces}

Let ${\mathcal X}$ be a finite subset of ${\mathcal P}(V)$ containing more than one element.
Denote by $S({\mathcal X})$ the group of all permutations on ${\mathcal X}$.
In this section we determine all cases when every element of $S({\mathcal X})$ can be extended to an element of ${\rm PGL}(V)$
if $R$ is a field.

We say that $P_{1},\dots,P_{m}\in {\mathcal P}(V)$ form an {\it independent} subset if
non-zero vectors
$x_{1}\in P_{1},\dots,x_{m}\in P_{m}$
are linearly independent.
Every permutation on an independent subset can be extended to an element of ${\rm PGL}(V)$.

Let $m\in \{2,\dots,n\}$.
An $(m+1)$-element subset ${\mathcal X}\subset {\mathcal P}(V)$ is called an $m$-{\it simplex}
if it is not independent and every $m$-element subset of ${\mathcal X}$ is independent.
For example, if $x_{1},\dots,x_{m}\in V$ are linearly independent and all 
$a_{1},\dots,a_{m}\in R$ are non-zero then
$$\langle x_{1}\rangle,\dots, \langle x_{m}\rangle\;\mbox{ and }\;\langle a_{1}x_{1}+\dots+a_{m}x_{m}\rangle$$
form an $m$-simplex.
Conversely, if $\{P_{1},\dots,P_{m+1}\}$ is an $m$-simplex then
there exist linearly independent vectors
$$x_{1}\in P_{1}\setminus\{0\},\dots,x_{m}\in P_{m}\setminus\{0\}\;
\mbox{ such that }\;
P_{m+1}=\langle x_{1}+\dots+x_{m}\rangle.$$
Every permutation on an $m$-simplex can be extended to an element of ${\rm PGL}(V)$
\cite[Section III.3, Proposition 1]{Baer}.

Following \cite[Section III.4, Remark 5]{Baer} we say that a subset ${\mathcal X}\subset {\mathcal P}(V)$ is {\it harmonic}
if there are linearly independent vectors $x,y\in V$ such that
$${\mathcal X}=\{\langle x \rangle,\langle y \rangle,\langle x+y\rangle,\langle x-y\rangle\}.$$

\begin{exmp}\label{exmp-p2-1}{\rm
Suppose that the characteristic of $R$ is equal to $3$ and ${\mathcal X}$ is the har\-mo\-nic subset
consisting of
$$P_{1}=\langle x \rangle,\;P_{2}=\langle y \rangle,\;P_{3}=\langle x+y\rangle,\;P_{4}=\langle x-y\rangle.$$
Consider $u_{1},u_{2},u_{3}\in {\rm GL}(V)$ satisfying the following conditions
$$\begin{array}{ll}
u_{1}(x)=y&u_{1}(y)=x,\\
u_{2}(x)=-x&u_{2}(y)=x+y,\\
u_{3}(x)=x&u_{3}(y)=-y.
\end{array}$$
Since the characteristic of $R$ is equal to $3$,  we have
$$u_{2}(x-y)=-x-(x+y)=-2x-y=x-y.$$
A direct verification shows  that every $\pi(u_{i})$
is an extension of the transposition $(P_{i},P_{i+1})$.
Since the group $S({\mathcal X})$ is spanned by all transpositions of type $(P_{i},P_{i+1})$,
every permutation on ${\mathcal X}$ can be extended to an element of  ${\rm PGL}(V)$.
}\end{exmp}

\begin{prop}\label{prop2}
Suppose that $R$ is a field.
If  every permutation on ${\mathcal X}$ can be extended to an element of ${\rm PGL}(V)$ then one of
the following possibilities is realized:
\begin{enumerate}
\item[$\bullet$] ${\mathcal X}$ is an independent subset;
\item[$\bullet$] ${\mathcal X}$ is an $m$-simplex, $m\in \{2,\dots,n\}$;
\item[$\bullet$] the characteristic of $R$ is equal to $3$ and ${\mathcal X}$ is a harmonic subset.
\end{enumerate}
\end{prop}

\begin{lemma}\label{lemma-pr2}
If $R$ is a field and $f$ is an element of ${\rm PGL}(V)$ transferring $P\in {\mathcal P}(V)$ to $Q\in {\mathcal P}(V)$
then for any non-zero vectors $x\in P$ and $y\in Q$ there exists $u\in {\rm GL}(V)$
such that $f=\pi(u)$ and $u(x)=y$.
\end{lemma}

\begin{proof}
We take any $v\in {\rm GL}(V)$ satisfying $f=\pi(v)$. Then $v(x)=ay$ and the linear automorphism
$u:=a^{-1}v$ is as required.
\end{proof}

\begin{rem}{\rm
If $R$ is non-commutative then
a scalar multiple of a linear mapping is linear only in the case when the scalar belongs to the center of $R$.
}\end{rem}

\begin{proof}[Proof of Proposition \ref{prop2}]
Let $P_{1},\dots,P_{k}$ be the elements of ${\mathcal X}$.
If ${\mathcal X}$ is not independent then we take any maximal independent subset in ${\mathcal X}$.
Suppose that it is formed by $P_{1},\dots,P_{m}$, $k<m$ and
consider $P_{p}$ with $p>m$. Every non-zero vector $y\in P_{p}$ is a linear combination of
non-zero vectors $y_{1}\in P_{1},\dots,y_{m}\in P_{m}$.
If this linear combination contains $y_{i}$ and does not contain $y_{j}$ for some $i,j\le m$ then
an element of ${\rm PGL}(V)$ extending the transposition $(P_{i},P_{j})$
does not leave fixed $P_{p}$ which is impossible.
This means that
$$y=a_{1}y_{1}+\dots+a_{m}y_{m},$$
where all $a_{1},\dots,a_{m}\in R$ are non-zero.

Thus $P_{1},\dots,P_{m}$ and $P_{p}$ form an $m$-simplex for every $p>m$.
If ${\mathcal X}$ consists of $m+1$ elements, i.e. $k=m+1$, then ${\mathcal X}$ is an $m$-simplex.
Consider the case when $k\ge m+2$.

We choose non-zero vectors $x_{1}\in P_{1},\dots,x_{m}\in P_{m}$ such that
$$x_{m+1}:=x_{1}+\dots+x_{m}\in P_{m+1}.$$
If $p\ge m+2$ then
$$P_{p}=\langle x_{p}\rangle,\;\mbox{ where }\; x_{p}= b_{1}x_{1}+\dots+b_{m}x_{m}$$
and all $b_{1},\dots,b_{m}\in R$ are non-zero.
Let $v$ be a linear  automorphism of $V$ such that $\pi(v)$ is an extension of the transposition $(P_{m+1},P_{p})$.
By Lemma \ref{lemma-pr2}, we can suppose that $v$ sends $x_{m+1}$ to $x_{p}$.
Since $x_{1},\dots,x_{m}$ are linearly independent and $v(P_{i})=P_{i}$ for every $i\le m$,
the equality
$$v(x_{1})+\dots+ v(x_{m})=v(x_{m+1})=x_{p}=b_{1}x_{1}+\dots+b_{m}x_{m}$$
shows that $v(x_{i})=b_{i}x_{i}$ for every $i\le m$.
Then
$$v(x_{p})=b_{1}v(x_{1})+\dots+b_{m}v(x_{m})=b^{2}_{1}x_{1}+\dots+b^{2}_{m}x_{m}\in P_{m+1}$$
which means that $b^{2}_{1}=b^{2}_{2}=\dots=b^{2}_{m}$
and $b_{i}=\pm b_{j}$ for any $i,j\le m$. In other words,
$$x_{p}=b(\varepsilon_{1}x_{1}+\dots+\varepsilon_{m}x_{m}),$$
where $\varepsilon_{i}=\pm 1$ for every $i\in \{1,\dots,m\}$.
Since $x_{m+1}$ and $x_{p}$ are linearly independent,
$\varepsilon_{i}\ne\varepsilon_{j}$ for some pairs $i,j\le m$.
This guarantees that the characteristic of $R$ is not equal to $2$
and we can assume that
$$P_{p}=\langle x_{p}\rangle,\; \mbox{ where }\;
x_{p}=x_{1}+\dots+x_{q}-x_{q+1}-\dots-x_{m}
$$
and $1\le q<m$.

Now consider a linear automorphism $u\in {\rm GL}(V)$ such that $\pi(u)$ is an extension of the transposition $(P_{q},P_{q+1})$.
Then
$$u(P_{q})=P_{q+1},\;u(P_{q+1})=P_{q}\;\mbox{ and }\;u(P_{i})=P_{i}\;\mbox{ if }\;i\ne q,q+1.$$
By Lemma \ref{lemma-pr2}, we suppose that $u$ leaves fixed $x_{m+1}$.
Since $x_{1},\dots,x_{m}$ are linearly independent,
the equality
$$u(x_{1})+\dots+ u(x_{m})=u(x_{m+1})=x_{m+1}=x_{1}+\dots+ x_{m}$$
implies that
$$u(x_{q})=x_{q+1},\;u(x_{q+1})=x_{q}\;\mbox{ and }\;u(x_{i})=x_{i}\;\mbox{ if }\;i\ne q,q+1,\;i\le m.$$
Then
$$u(x_{p})=u(x_{1})+\dots+u(x_{q})-u(x_{q+1})-\dots-u(x_{m})$$
belongs to $P_{p}$ only in the case when  $q=1$ and $m=2$, i.e.
$P_{p}=\langle x_{1}-x_{2}\rangle$ for every $p\ge 4$.
The latter means that ${\mathcal X}$ is the harmonic subset consisting of
$$P_{1}=\langle x_{1} \rangle,\;P_{2}=\langle x_{2} \rangle,\;P_{3}=\langle x_{1}+x_{2}\rangle,
\;P_{4}=\langle x_{1}-x_{2}\rangle.$$

Let $w$ be a linear automorphism of $V$ such that $\pi(w)$ is an extension of the transposition $(P_{1},P_{3})$
and $w(x_{1})=x_{1}+x_{2}$ (see Lemma \ref{lemma-pr2}).
Since $w(P_{2})=P_{2}$ and $w(P_{3})=P_{1}$, we have
$$w(x_{1}+x_{2})=w(x_{1})+w(x_{2})=(x_{1}+x_{2})+cx_{2}\in P_{1}.$$
Then $c=-1$ and $w(x_{2})=-x_{2}$.
The equality $w(P_{4})=P_{4}$ implies that
$$w(x_{1}-x_{2})=w(x_{1})-w(x_{2})=(x_{1}+x_{2})+x_{2}=x_{1}+2x_{2}\in P_{4}.$$
Hence $x_{1}+2x_{2}=x_{1}-x_{2}$ and $2=-1$, i.e. the characteristic of $R$ is equal to $3$.
\end{proof}

Also, we will need the following simple lemma concerning harmonic subsets.

\begin{lemma}\label{lemma-harm}
If the intersection of two harmonic subsets contains three elements
then these harmonic subsets are coincident.
\end{lemma}

\begin{proof}
See, for example, \cite[Section III.4, Proposition 2]{Baer}.
\end{proof}

\section{Proof of Theorem \ref{theorem-main2}}
Let $n\ge 2$ and let $f:{\mathcal P}(V)\to {\mathcal P}(V')$ be a non-constant PGL-mapping.

We say that a subspace $S\subset V$ is {\it invariant} for $h\in {\rm PGL}(V)$ if $h$
transfers ${\mathcal P}(S)$ to itself, in other words, $S$ is invariant for every $u\in {\rm GL}(V)$
satisfying $h=\pi(u)$.

\begin{lemma}\label{lemma-t2-1}
The following assertions are fulfilled:
\begin{enumerate}
\item[{\rm (1)}] for every $h\in {\rm PGL}(V)$ there exists unique $\overline{h}\in {\rm PGL}(V_{g})$ such that
$fh={\overline h}f$,
\item[{\rm (2)}] the mapping $f$ is injective,
\item[{\rm (3)}] the mapping $h\to \overline{h}$ is a monomorphism of ${\rm PGL}(V)$ to ${\rm PGL}(V_{f})$.
\end{enumerate}
\end{lemma}

\begin{proof}
(1). Let $h\in {\rm PGL}(V)$. We take any $h'\in {\rm PGL}(V')$ such that $fh=h'f$.
As in the proof of the first part of Lemma \ref{lemma-t1-1},
we show that $V_{f}$ is invariant for $h'$ and
the equality $h'|_{{\mathcal P}(V_{f})}=h''|_{{\mathcal P}(V_{f})}$ holds for any $h''\in {\rm PGL}(V')$ satisfying $fh=h''f$.

(2).
Suppose that $f(P)=f(Q)$ for some distinct $P,Q\in {\mathcal P}(V)$.
Then
$$fh(P)={\overline h}f(P)={\overline h}f(Q)=fh(Q)\;\;\;\;\;\forall\;h\in {\rm PGL}(V).$$
For every $T\in {\mathcal P}(V)$ there exists $h\in {\rm PGL}(V)$ leaving fixed $P$ and transferring $Q$ to $T$.
Then
$$f(P)=fh(P)=fh(Q)=f(T)$$
which means that $f$ is constant, a contradiction.

(3). As in the proof of the second part of Lemma \ref{lemma-t1-1}, 
we establish that the mapping $h\to \overline{h}$ is a homomorphism of ${\rm PGL}(V)$ to ${\rm PGL}(V_{f})$.
If $\overline{h}$ is the identity element of ${\rm PGL}(V_{f})$ then
$$fh(P)={\overline h}f(P)=f(P)\;\;\;\;\;\forall\;P\in {\mathcal P}(V).$$
Since $f$ is injective, the latter equality implies that 
$h(P)=P$ for every $P\in {\mathcal P}(V)$.
\end{proof}

\begin{lemma}\label{lemma-t2-2}
Let $h\in {\rm PGL}(V)$ and let $S$ be a subspace of $V$.
The following assertions are fulfilled:
\begin{enumerate}
\item[{\rm (1)}] if $S$ is invariant for $h$ then $S_{f}$ is invariant for ${\overline h}$,
\item[{\rm (2)}] if $h|_{{\mathcal P}(S)}$ is identity then ${\overline h}|_{{\mathcal P}(S_{f})}$ is identity.
\end{enumerate}
\end{lemma}

\begin{proof}
Similar to the proof of Lemma \ref{lemma-t1-3}.
\end{proof}

From this moment we suppose that $R'$ is a field. 
We use Proposition \ref{prop2} to prove the following.

\begin{lemma}\label{lemma-t2-3}
The mapping $f$ transfers any independent subset consisting of $n-1$ elements to an  independent subset.
\end{lemma}

\begin{proof}
Let ${\mathcal X}\subset {\mathcal P}(V)$ be an independent subset consisting of $n-1$ elements
and let ${\mathcal Y}$ be a maximal independent subset of ${\mathcal P}(V)$ containing ${\mathcal X}$.
Note that ${\mathcal Y}$ consists of $n$ elements and every permutation on ${\mathcal Y}$
can be extended to an element of ${\rm PGL}(V)$.
As in the proof of Lemma \ref{lemma-t1-5}, 
we show that every permutation on $f({\mathcal Y})$ is extendable to an element of ${\rm PGL}(V_{f})$.
By Proposition \ref{prop2}, $f({\mathcal Y})$ is an independent subset or an $(n-1)$-simplex or a harmonic subset.
In the first and second cases,
every $(n-1)$-element subset of $f({\mathcal Y})$ is independent which implies that $f({\mathcal X})$ is independent.
Show that the third possibility is not realized.

Suppose that $f({\mathcal Y})$ is a harmonic subset. Then $n=4$ and $f({\mathcal X})$ is a $2$-simplex.
We take any maximal independent subset  ${\mathcal Y}'$ containing ${\mathcal X}$
and distinct from ${\mathcal Y}$.
Then $f({\mathcal Y}')$ is an independent subset or a $3$-simplex or a harmonic subset.
Since $f({\mathcal Y}')$ contains the $2$-simplex $f({\mathcal X})$,
the first and second possibilities can not be realized.
So, $f({\mathcal Y}')$ is a harmonic subset.
The intersection of the harmonic subsets $f({\mathcal Y})$ and $f({\mathcal Y}')$ coincides with $f({\mathcal X})$,
hence it contains three elements. By  Lemma \ref{lemma-harm}, we have $f({\mathcal Y})=f({\mathcal Y}')$ which contradicts
the injectivity of $f$.
\end{proof}

Suppose that $\dim V_{f}\le n$.
Using Lemma \ref{lemma-t2-3}, the second part of Lemma \ref{lemma-t2-2}
and the fact that the mapping $h\to {\overline h}$ is a monomorphism of ${\rm PGL}(V)$ to ${\rm PGL}(V_{f})$, 
we establish the following:
$\dim V_{f}=n$ and for every $(n-1)$-dimensional subspace $S\subset V$
the subspace $S_{f}$ is $(n-1)$-dimensional (the proof is similar to the proof of Lemma \ref{lemma-t1-6}).

Let $n\ge 3$. 
By the first part of Lemma \ref{lemma-t2-2}, for every subspace $S\subset V$ 
(we assume that $\dim S\ge 2$)
the restriction of $f$ to ${\mathcal P}(S)$ is a PGL-mapping.
As in the proof of Lemma \ref{lemma-t1-7}, we show that $\dim S_{f}=\dim S$
for every non-zero subspace $S\subset V$.
Then $f$ satisfies the conditions of Theorem \ref{theoremFFH}
and there is a semilinear injection $l:V\to V'$ such that $f=\pi(l)$.
Since $\dim V_{l}=\dim V_{f}=n$, the mapping $l$ sends bases of $V$ to bases of $V_{f}$
which means that $l$ is a strong semilinear embedding.

\section{Examples}
The following example shows that the condition $\dim V_{g}\le n$ in Theorem \ref{theorem-main1}
cannot be omitted.
\begin{exmp}\label{exmp-t1-1}{\rm
Let $l:V\to V'$ be a strong semilinear embedding.
Suppose that $\dim V'>n$ and take any vector $y\in V'$ which does not belong to 
the $n$-dimensional subspace $V_{l}$.
Consider the mapping $g:V\to V'$ defined as follows
$$g(x)=
\begin{cases}
l(x)&x\ne 0\\
y &x=0\,.
\end{cases}$$
Then $\dim V_{g}=n+1$ and $g$ is a GL-mapping.
Indeed, for every $u\in {\rm GL}(V)$ there exists $\overline{u}\in {\rm GL}(V_{l})$ satisfying $lu=\overline{u}l$
and we have $gu=u'g$ for every $u'\in {\rm GL}(V')$ such that $u'|_{V_{l}}=\overline{u}$ and $u'(y)=y$.
}\end{exmp}

There are other non-trivial GL-mappings which are not strong semilinear embeddings.

\begin{exmp}[H. Havlicek]\label{exmp-t1-2}{\rm
Suppose that $R$ is a finite field and $\dim V'=|V|$. 
Consider any bijection $g$ of $V$ to a base $B'$ of $V'$.
If $u$ is a bijective transformation of $V$ then $gug^{-1}$ is a permutation on $B'$.
Since $B'$ is a base of $V'$, we can extend $gug^{-1}$ to a certain $u'\in {\rm GL}(V')$.
Then $gu=u'g$. Therefore, $g$ is a GL-mapping.
The dimension of $V_{g}$ is greater than $n$.
}\end{exmp}

We modify the latter example and construct non-constant PGL-mappings which 
are not induced by strong semilinear embeddings.

\begin{exmp}\label{exmp-t2-1}{\rm
Suppose that $R$ is a finite field and $\dim V'\ge m=|{\mathcal P}(V)|$. 
Consider any bijection $f$ of ${\mathcal P}(V)$ to an independent subset or an $m$-simplex of ${\mathcal P}(V')$.
As in Example \ref{exmp-t1-2}, we show that $f$ is a PGL-mapping.
Note that  $\dim V_{f}>n+1$ if $n\ge 3$. 
So, this example does not show that the condition 
$\dim V_{f}\le n$ in Theorem \ref{theorem-main2} cannot be omitted.
}\end{exmp}

\subsection*{Acknowledgement}
I express my deep gratitude to 
Hans Havlicek for Example \ref{exmp-t1-2}.
Also, I am  grateful to Andrea Blunck and Warren Dicks for information
and Andrzej Matras and Jaroslaw Kosiorek for discussion.

\end{document}